\documentclass[11pt]{article}
\title{Gaussian limits for traces of circular beta-ensembles}
\date{}
\author{Yun Li}

\usepackage{amsmath, amsthm, amssymb}
\usepackage{graphicx,bm}
\usepackage{amsmath, enumerate}
\usepackage{color, url}
\usepackage{hyperref}
    \newtheorem{theorem}{Theorem}
    \newtheorem{thmx}{Theorem}

    \newtheorem{proposition}[theorem]{Proposition}
    
    \newtheorem{claim}[theorem]{Claim}

\theoremstyle{definition} % For roman text in the body
    \newtheorem{definition}[theorem]{Definition}
    \newtheorem{remark}[theorem]{Remark}
    
    \newtheorem{assumption}{Assumption}

\newcommand{\R}{{\mathbb R}}
\newcommand{\CC}{{\mathbb C}}

\newcommand{\ev}{{\rm   E}}
\newcommand{\pr}{\mbox{\rm P}}
\newcommand{\lstar}{{\raise-0.15ex\hbox{$\scriptstyle \ast$}}}

\theoremstyle{remark} % For an italic header, more subtle than definition style

\definecolor{violet}{rgb}{0.8,0,0.2}

\newcommand{\ed}{\stackrel{d}{=}}

\newcommand{\cC}{{\mathcal C}}

\newcommand{\D}{{\mathbb D}}

\newcommand{\cD}{{\mathcal D}}

\newcommand{\cL}{{\mathcal L}}
\newcommand{\cM}{{\mathcal M}}

\newcommand{\mat}[4]{\left( \hspace{-0.4em} \begin{array}{cc}
#1 & #2  \\
#3 & #4  \\
\end{array} \hspace{-0.4em}\right)}

\newcommand{\tr}{\operatorname{tr}}
\newcommand{\ind}{\mathbf 1}
\newcommand{\ol}[1]{\overline{#1} \!\,} %overline

\definecolor{darkgreen}{rgb}{.2,0.4,0.3}

\newcommand{\tl}{\tilde}

\newcommand{\spec}{\operatorname{spec}}

\newcommand{\Circ}{\mathsf{Circ}}

\begin{document}
\maketitle

\begin{abstract}
    Consider the circular beta-ensembles, Jiang and Matsumoto computed the moments of traces, and proved central limit theorems for polynomial linear statistics using Jack functions. 
    We give a direct proof of the Gaussian limits for the traces of powers of circular beta-ensembles using the theory of orthogonal polynomials on the unit circle and Stein's method. The proof relies on an explicit coupling of the Verblunsky coefficients of circular beta-ensembles, which also allow us to obtain almost sure and $L^2$ convergence.  
\end{abstract}

\section{Introduction}
The size $n$ circular beta-ensemble with parameter $\beta>0$ is the distribution of $n$ points $\{e^{i\theta_j}, 1\le j\le n\}$ on the unit circle with  joint probability density
\begin{equation}\label{eq:cbe}
\frac{1}{Z_{n,\beta}}\prod_{j<k\le n}|e^{i\theta_j}-e^{i\theta_k}|^\beta, \qquad \theta_j\in [-\pi,\pi),
\end{equation}
where $Z_{n,\beta}$ is an explicitly normalizing constant.
When $\beta=2$, this recovers the joint eigenvalue density function of a Haar unitary matrix. Due to the structure of the squared Vandermonde, this ensemble is determinantal \cite{AGZ,HKPV}, i.e.~the finite dimensional marginal densities can be expressed as determinants of certain kernel function. The case when $\beta=1,4$ correspond to symmetric/self-dual unitary matrices. 

For general $\beta>0$, these random points $\{e^{i\theta_j},1\le j\le n\}$ can be viewed as one-component log-gas confined on the unit circle at temperature $\beta^{-1}$. There is no natural invariant matrix model for general $\beta$. However, using the theory of orthogonal polynomials on the unit circle (OPUC), Killip and Nenciu \cite{KillipNenciu} constructed a family of sparse random unitary matrix models $\{\Circ_{n,\beta},n\ge 1\}$ in terms of the so-called Verblunsky coefficients. The idea can be summarized as follows. 

Suppose that $\mu$ is a discrete probability  measure on the unit circle $\partial \D$ with a finite support of $n$ points. Then $\mu$ is encoded with its monic orthogonal polynomials, which satisfy the famous Szeg\H{o} recursion (see \eqref{eq:Szego1} below). The recursion can be parameterized with the Verblunsky coefficients, a finite collection of complex numbers $\alpha_0, \dots, \alpha_{n-1}$. In \cite{CMV} Cantero, Moral, and Velasquez provided a construction for a sparse $n\times n$ unitary matrix (called the {CMV matrix})
\[
\cC=\cC(\alpha_0,\dots, \alpha_{n-1}),
\]
in terms of the Verblunsky coefficients, so that the spectral measure  of $\cC$ with respect to the unit vector $\mathbf{e}_1=(1,0,\dots,0)^{\dag}$ is exactly $\mu$. 
Using this idea, together with a Jacobian computation between the spectral information of $\mu$ and the Verblunsky coefficients, Killip and Nenciu proved that the joint eigenvalue distribution of $\Circ_{n,\beta}$ is given by \eqref{eq:cbe}.
So we will still refer to these random particles $\{e^{i\theta_j},1\le j\le n\}$ as eigenvalues, and we write $\Lambda_n=\{\lambda_1,\dots,\lambda_n\}\sim \text{C$\beta$E}_n$ if the joint density is  \eqref{eq:cbe}. 

In \cite{JM15} Jiang and Matsumoto derived the central limit theorems (CLTs) for the linear statistics of polynomials of circular beta-ensemble. The $\beta=2$ case was proved by Diaconis and Shahshahani \cite{DS94} using the connection to representation theory. Throughout the paper, we denote $Z\sim\CC N(0,\sigma^2)$ if the real and complex parts of $Z$ are independent real normal random variables with mean $0$ and variance $\sigma^2/2$. 
\begin{thmx}[\cite{DS94,JM15}]\label{thm:clt1}
    Fix $\beta>0$ and let $\Lambda_n=\{\lambda_1,\dots,\lambda_n\}\sim \textup{C$\beta$E}_n$. For any fixed $k$, the random vector
    \begin{equation}\label{eq:clt}
    \left(\sum_{i=1}^n\lambda_i, \sum_{i=1}^n\lambda_i^2, \cdots \sum_{i=1}^n\lambda_i^k\right)\Rightarrow \left(\xi_1,\xi_2,\dots,\xi_k\right)
    \end{equation}
    weakly as $n\to\infty$, where $\xi_j$'s are independent complex normal random variables with $\xi_j\sim\CC N(0,\frac{2}{\beta}j)$ for each $j$. 
\end{thmx}
The proof of Theorem \ref{thm:clt1} relies on a careful analysis of the moments of traces of circular beta-ensemble using Jack polynomials. Note that Jiang and Matsumoto also derived CLTs for general testing functions when $\beta=1,4$. See \cite{DiaconisEvans} for CLTs for the eigenvalues arising from classical orthogonal, unitary and symplectic groups. 

As a consequence of these CLTs, one can show that the eigenvalues of the circular beta-ensemble are log-correlated. More precisely, define the (normalized) characteristic polynomial of circular beta-ensemble as 
\begin{align*}
    X_n(z):=\prod_{j=1}^n (1-ze^{-i\theta_j}),\qquad \{e^{i\theta_j},1\le j\le n\}\sim\text{C$\beta$E}_n.
\end{align*}
The real (and imaginary) part of $\log X_n$ on the unit circle converges weakly to a random Gaussian field with correlation kernel proportional to $\log |e^{i\omega_1}-e^{i\omega_2}|^{-1}$. The parameter $\beta$ describes the strength of the correlation kernel. We refer to \cite{Jo88,HKO01,Lambert_CLT} for more details. In \cite{Virag14}, Vir\'ag raised the question on the connection between Gaussian multiplicative cascades and random matrix theory.
Recently in \cite{CN_GMC}, Chhaibi and Najnudel established an equality between the projective limit of the circular beta-ensemble and Kahane’s Gaussian Multiplicative Chaos (GMC) on the unit circle (after a proper normalization) in the subcritical and critical regimes. Using the integrability of the circular beta-ensemble, Lambert and Najnudel \cite{LambertNajnudel} derived the multiplicative chaos convergence for the characteristic polynomial in the subcritical regime. See e.g.~\cite{RhodesVargas} for a survey on the GMC theory. 
See also \cite{KS,BVBV_op} for the local scaling limit of the circular beta-ensemble, and \cite{CNN,BVBV_szeta} for the microscopic limit of the normalized characteristic polynomial. 
 
The starting point of the proof of \cite{CN_GMC} is a coupling of circular beta-ensembles of all dimensions through the Verblunsky coefficients:
let $\alpha_j,j\ge 0$ be independent rotational invariant random variables such that $|\alpha_j|^2\sim \textup{Beta}(1,\tfrac{\beta}{2}(j+1))$, and let $\eta$ be uniformly chosen on the unit circle, independent of $\alpha_j,j\ge 0$. Then Killip and Nenciu \cite{KillipNenciu} proved that
\[
\spec\{\cC(\alpha_0,\alpha_1,\dots,\alpha_{n-2},\eta)\}\ed \spec\{\Circ_{n,\beta}\} \sim \text{C$\beta$E}_n.
\]
It is then possible to couple C$\beta$E$_n$ and C$\beta$E$_{n+1}$ in the same probability space so that the first $(n-1)$ Verblunsky coefficients match exactly. Another key ingredient is the identification of the limit of $\log X_n$ as the (log-correlated) Gaussian analytic function. By comparing the corresponding Taylor coefficients, Chhaibi and Najnudel obtained the following non-trivial identity in law:
\begin{equation}\label{eq:1st}
            \sum_{j\ge -1}\alpha_j\bar \alpha_{j+1} \sim \sqrt{\tfrac{2}{\beta}} \CC N(0,1),
\end{equation}
where we define $\alpha_{-1}=-1$ by convention. This convention is consistent with that in \cite{OPUC1} and will be used throughout the paper. 
% \begin{theorem}[\cite{JM15}]\label{eq:1st}
%     Fix $\beta>0$, let $\alpha_j,j\ge 0$ be independent rotational invariant random variables such that $|\alpha_j|^2\sim \textup{Beta}(1,\tfrac{\beta}{2}(j+1))$, then
%     \begin{align}
%     \end{align}
%     where we define $\alpha_{-1}=-1$ by convention, and denote by $Z\sim \CC N(0,1)$ if the real and complex parts of $Z$ are i.i.d.~$N(0,\tfrac12)$. 
% \end{theorem}

In terms of CMV matrices, we have $\sum_i \lambda_i^k \ed \tr (\Circ_{n,\beta}^k)\ed \tr (\cC(\alpha_0,\alpha_1,\cdots,\alpha_{n-2},\eta)^k)$. This motivates us to find a direct proof of the convergence of the traces of powers of the circular beta-ensembles under the coupling described above. This coupling also allows us to derive the a.s.~and $L^2$ convergence of \eqref{eq:clt}. 
\begin{theorem}\label{thm:main}
    Fix $\beta>0$. There exists a coupling of a sequence of random unitary matrices $\cC_n,n\ge 1$ such that $ \spec\{\cC_n\}\ed \spec\{\Circ_{n,\beta}\} \sim \text{C$\beta$E}_n$, 
    and for any fixed $k\ge 1$ we have 
    \begin{equation}\label{eq:conv_vec}  \left(\tr(\cC_n),\tr(\cC_n^2),\cdots,\tr(\cC_n^k)\right)\to \left(T_1,T_2,\dots,T_k\right) \quad \text{a.s. and in $L^2$ as $n\to\infty$,}
    \end{equation}
    where $T_j, j\ge 1$ is a sequence of independent complex normal random variables with $T_j\sim\CC N(0,\frac{2}{\beta}j)$. 
\end{theorem}
The proof of Theorem \ref{thm:main} relies on applying an analogue of Stein's method to the sequence of Verblunsky coefficients, and an analysis of the Ablowitz-Ladik system for the circular beta-ensemble. Along the way, we also find a direct proof of \eqref{eq:1st}. 
\begin{theorem}\label{thm:1st}
    Under the coupling in Theorem \ref{thm:main}, the random variable $T_j$ can be represented as the trace of the $j$-th power of an infinite CMV matrix $\cC_\infty$. In particular, 
    \begin{align*}
        T_1 = \sum_{j\ge -1}\alpha_j\bar\alpha_{j+1} \sim \CC N(0,\tfrac{2}{\beta}),
    \end{align*}
    where $\alpha_j,j\ge 0$ is a sequence of independent, rotational invariant random variables such that $|\alpha_j|^2\sim \textup{Beta}(1,\tfrac{\beta}{2}(j+1))$.
\end{theorem}

The proofs of Theorems \ref{thm:main} and \ref{thm:1st} consist of two main steps. First, we establish the convergence of the traces of powers of circular beta-ensembles to the corresponding traces of powers of the infinite CMV matrix described in Theorem \ref{thm:1st}. Then we identify the distribution of the limiting objects. Although parts of the first step are known in the literature (see e.g.~\cite{CN_GMC}), we include the proofs together with some elementary estimates for completeness. In forthcoming joint work with Liu and Najnudel \cite{LLN}, we provide a more refined analysis, see Remark \ref{rmk:LLN} below for more details.

An advantage of our method is that it establishes both almost surely and $L^2$ convergence under the natural coupling of circular beta-ensembles of all dimensions. Our main contribution is the use of Stein's method together with the Ablowitz-Ladik system to characterize the distributions of the limits.
Heuristically speaking, the traces of powers of the limiting CMV matrix can be viewed as eigenvectors of certain differential operators associated with the Ablowitz-Ladik system, see Remark \ref{rmk:infinite} below for further discussion. The Gaussian limits then follow directly from Stein's method, and certain identities of Poisson bracket related to the Ablowitz-Ladik system.
We treat the $k=1$ case separately to illustrate the main idea and to provide a direct proof of \ref{eq:1st}, as this case may also be of independent interest.

We conclude the introduction with a brief review of the related work. For circular beta-ensembles, Webb \cite{Webb_CLT} generalized the result of \cite{JM15} using Stein’s method and obtained quantitative convergence rantes. Stein's method has also been used to establish quantitative CLTs for a growing number of traces for classical compact groups \cite{DoblerStolz}. See \cite{DE_CLT,DP_CLT} for CLTs for other beta-ensembles, as well as the quantitative results \cite{lambert2019quantitative}. Another approach to establishing CLTs for random matrices is based on loop equations. This idea goes back to Johansson's work \cite{Jo_loop} and has been developed further in \cite{BorotGuionnet,Shc13}. See \cite{Lambert_CLT} for mesoscopic CLT for circular beta-ensemble. 
Johansson \cite{Johansson97Annals} established a super-exponential rate of convergence to the Gaussians for  Haar unitary matrices, and later Johansson and Lambert \cite{JohanssonLambert} obtained
multivariate bounds in total variation that allow for a growing number of traces.

We expect that our framework can be adapted to study CLTs for other beta-ensembles on the circle and the real line, including the circular Jacobi beta-ensemble and the real orthogonal beta-ensemble. To keep the present paper concise, we leave these extensions for future work.

\medskip

\noindent \textbf{Outline of the paper:}
In Section \ref{sec:background}, we provide the necessary background on OPUC theory, CMV matrices, Stein's method and Ablowitz-Ladik system. Section \ref{sec:1st} contains a direct proof of the distributional identity \eqref{eq:1st}. In Section \ref{sec:proof} we complete the proofs of Theorems \ref{thm:main} and \ref{thm:1st}. Section \ref{sec:Appendix} (the Appendix) collects some of the computations for completeness. 
\medskip

\noindent \textbf{Acknowledgment:}
The author thanks B\'alint Vir\'ag for pointing out the connection to Schur flow and the Ablowitz-Ladik system. The author also thanks Reda Chhaibi and Benedek Valk\'o for valuable discussions. The author is grateful for the hospitality of the Chinese Academy of Sciences and Institut Mittag-Leffler, where the proof ideas were conceived and completed during the \textit{RMT-GMC seminars} and the workshop on \textit{Characteristic Polynomials}.

\section{Preliminaries}\label{sec:background}

\subsection{OPUC and CMV matrices}\label{subsec:OPUC}
We provide a brief review of OPUC theory, together with some properties of CMV matrices. We refer to \cite{OPUC1} for more details. 

Let $\nu$ be a probability measure on $\partial \D$, and let $\Phi_k,k\ge 0$ be the Gram-Schmidt orthogonalization of the polynomials $1,z,z^2,\cdots$. 
Define the reversed polynomials  
\begin{align*}
\Phi_k^*(z):=z^k \overline{\Phi_k(1/\bar z)}.
\end{align*}
Then the famous Szeg\H o recursion (see e.g.~Section 1.5 of \cite{OPUC1}) states that 
\begin{align}\label{eq:Szego1}
\binom{\Phi_{k+1}}{\Phi_{k+1}^*}=
% A_k Z \binom{\Phi_{k}}{\Phi_{k}^*}, \qquad \binom{\Phi_0}{\Phi_0^*}
\mat{1}{-\bar \alpha_k}{-\alpha_k}{1} \mat{z}{0}{0}{1} \binom{\Phi_{k}}{\Phi_{k}^*}, \qquad \binom{\Phi_0}{\Phi_0^*}
=\binom{1}{1},\qquad 0\le k\le n-1.
\end{align}
The constants $\alpha_k$, $0\le k\le n-1$ are called the Verblunsky coefficients. In the case when $\mu$ is supported on $n$ points on $\partial \D$, we have $|\alpha_j|<1$ for $0\le j\le n-2$ and $|\alpha_{n-1}|=1$. Otherwise, we have $|\alpha_j|<1$ for all $j\ge 0$. The Szeg\H o recursion can be viewed as the analogue of the three-term recursion for orthogonal polynomials on the line $\R$. Let us also note that  if $\mu$ determines the sequence of Verblunsky coefficients, then the converse is also true, see e.g.~Theorem 1.7.11 of \cite{OPUC1}.
% The map between the probability measures on $\partial \D$ and the possible Verblunsky coefficients is invertible, and both the map and its inverse are continuous. (See e.g.~Theorem 1.7.11 of \cite{OPUC1}.)

The next definition introduces the finite and infinite CMV matrices.

 \begin{definition}\label{def:CMV}
 Let $\{\alpha_0,\alpha_1,\cdots\}$ be a sequence of complex coefficients with  $|\alpha_k|\le 1$. Define
\[
\Xi_k = \begin{pmatrix}
\bar\alpha_k & \rho_k\\ \rho_k& -\alpha_k
\end{pmatrix},\qquad \rho_k=\sqrt{1-|\alpha_k|^2},
\]
and set $\Xi_{-1}=(1)$. Then we define the {CMV matrix} corresponding to $\{\alpha_0,\alpha_1,\cdots\}$ as
\begin{align}\label{eq:CMV}
	\cC(\alpha_0,\alpha_1,\cdots) :=\cL\cM,
\end{align}
where $\cL, \cM$ are block-diagonal matrices 
\begin{align*}
\cL=\operatorname{diag}\left(\Xi_0,\Xi_2\cdots\right),\qquad \cM=\operatorname{diag}\left(\Xi_{-1},\Xi_1\cdots\right).
\end{align*}
In the case when $|\{\alpha_0,\alpha_1,\cdots\}|=n$, we define $\Xi_{n-1}=(\bar\alpha_{n-1})$ to be $1\times 1$ matrix and set $\cL,\cM$ to be $n\times n$ matrices
\begin{align*}
\cL=\operatorname{diag}\left(\Xi_0,\Xi_2\cdots,\Xi_{2\lfloor\tfrac{n-1}{2}\rfloor}\right),\qquad \cM=\operatorname{diag}\left(\Xi_{-1},\Xi_1\cdots,\Xi_{2\lfloor\tfrac{n}{2}\rfloor-1}\right).
\end{align*}
 \end{definition}

Note that $\cC(\alpha_0,\dots, \alpha_{n-1})$ is unitary if and only if $|\alpha_{n-1}|=1$. In the case when there are infinitely many Verblunsky coefficients, the resulting matrix will be infinite of the form 
\begin{align*}
    \cC_\infty=\begin{pmatrix}
        \bar \alpha_0 & \rho_0\bar\alpha_1 &\rho_0 \rho_1&0 & 0 &\cdots\\
        \rho_0 & -\alpha_0\bar\alpha_1 &-\alpha_0\rho_1 &0&0&\cdots\\
        0&\rho_1\bar\alpha_2 & -\alpha_1\bar\alpha_2 &\rho_2\bar\alpha_3 &\rho_2\rho_3&\cdots\\
        0&\rho_1\rho_2 &-\alpha_1\rho_2 & -\alpha_2\bar\alpha_3 &-\alpha_2\rho_3 & \cdots\\
        0&0&0&\rho_3\bar\alpha_4 &-\alpha_3\bar\alpha_4 &\cdots\\
        \vdots &\vdots & \vdots &\vdots &\vdots &\ddots
    \end{pmatrix}.
\end{align*}
In particular, the trace of $\cC_\infty$ is given by $T_1=-\sum_{j=-1}^{\infty} \alpha_j\bar\alpha_{j+1}$ if the limit exists. 
The following proposition provides a crucial connection between CMV matrices and orthogonal polynomials of a discrete probability measure on $\partial \D$.

\begin{proposition}[\cite{KK,OPUC1}]\label{prop:CMV1}
Let $\cC$ be an infinite or finite CMV matrix corresponding to the sequence of Verblunsky coefficients $(\alpha_0,\alpha_1,\cdots)$, and let $\mu$ be the corresponding probability measure on $\partial\D$. 
Let $\Phi_k$ be the monic orthogonal polynomial associated to $\mu$ (or equivalent, to $\cC$), then 
    \begin{equation}\label{eq:CMV_characteristic}
    \Phi_k(z)=\det(zI_k-\mathcal{C}(\alpha_0,\dots,\alpha_{k-1})),
\end{equation}
for any $k\ge 0$ if $\cC$ is infinite, and for $0\le k\le n$ if $\cC$ has dimension $n$. 
\end{proposition}

% Suppose  that $U$ is an $n\times n$ unitary matrix for which  $\mathbf{e}_1=(1,0,\dots,0)^{\dag}$ is cyclic. The spectral measure of $U$ with respect to $\mathbf{e}_1$ is  given by the following discrete probability measure:
% \begin{equation}\label{eq:spectral_measure}
% \mu=\sum_{k=1}^n \delta_{\lambda_k} \cdot \left|\langle \mathbf{e}_1,  \mathbf{v}_k\rangle\right|^2.
% \end{equation}
% Here $\lambda_k,1\le k\le n$ are the eigenvalues of $U$, and $\mathbf{v}_k, 1\le k\le n$ are the corresponding unit length (right) eigenvectors. 
% The following proposition summarizes how a unitary matrix and the CMV matrix of its spectral measure are connected. 

% \begin{proposition}[\cite{CMV}]\label{prop:CMV2}
% Suppose that $U$ is an $n\times n$ unitary matrix for which $\mathbf{e}_1$ is cyclic. Let the spectral measure of $U$ with respect to $\mathbf{e}_1$ be $\nu$, and denote the Verblunsky coefficients of $\nu$ by $\alpha_k, 0\le k\le n-1$. Then there is a unitary matrix $V$ with $V\mathbf{e}_1=\mathbf{e}_1$ and $U=V \cC(\alpha_0,\dots, \alpha_{n-1}) V^{-1}$. In particular, $U$ and 
% $\cC$ are unitary equivalent, and they have the same spectral measure with respect to $\mathbf{e}_1$.
% \end{proposition}

Next, we review the Killip-Nenciu construction of the matrix models for circular beta-ensembles.
\begin{proposition}[\cite{KillipNenciu}]\label{prop:KN1}
    Consider the random probability measure  $\mu^{\textup{KN}}_{n,\beta} =\sum_{j=1}^n \pi_j\delta_{e^{i\theta_j}}$ on $\partial \mathbb{D}$,
where $\{e^{i\theta_j},1\le j\le n\}\sim \textup{C}\beta\textup{E}_n$, and the weights $\pi_j, 1\le j\le n$ are
chosen according to \textup{Dirichlet}$(\beta/2,\dots,\beta/2)$ distribution, independently of the support. 
Then Verblunsky coefficients $\tl\alpha_0,\tl\alpha_1,\dots,\tl\alpha_{n-1}$ corresponding to $\mu_n^{\textup{KN}}$ are independent, rotational invariant, such that $|\tl\alpha_j|^2\sim \textup{Beta}(1,\frac{\beta}{2}(n-j-1))$ for $0\le j\le n-2$, and $\tl \alpha_{n-1}\sim \textup{Unif}(\partial\D)$. 
\end{proposition}
The matrix model $\Circ_{n,\beta}$ is defined by $\cC(\tilde\alpha_0,\cdots,\tl\alpha_{n-1})$. Killip and Nenciu also proved the following result in \cite{KillipNenciu}. 
\begin{proposition}[\cite{KillipNenciu}]\label{prop:KN2}
    Let $ \alpha_j,j\ge 0$ be independent, rotational invariant random variables such that $| \alpha_j|^2\sim\textup{Beta}(1,\beta_j)$, where $\beta_j:=\frac{\beta}{2}(j+1)$. Let $\eta\sim\textup{Unif}(\partial\mathbb{D})$ be independent of $\alpha_j,j\ge 0$. Then the probability measure $\mu_n$ corresponding to $(\alpha_0,\dots,\alpha_{n-2},\eta)$ has support given by the size $n$ circular beta-ensemble.
\end{proposition} 
Note that the spectral weights of $\mu_n$ in Proposition \ref{prop:KN2} are no longer Dirichlet-distributed. The advantage of the reversed ordering of $\alpha_j$ in Proposition \ref{prop:KN2} is that the Verblunsky coefficients and the measures $\mu_n$ are now coupled in the same probability space. This reversed ordering has also been used in the study of the so-called truncated circular beta-ensembles, see \cite{KK,LV_truncated,LLN}. From now on, we work with the following probability space.
\begin{assumption}\label{coupling}
    Let $\{\alpha_j,j\ge 0\}$  be a sequence of independent, rotational invariant random variables such that 
\[
|\alpha_j|^2\sim \textup{Beta}(1,\beta_j),\qquad \beta_j=\tfrac{\beta}{2}(j+1)\qquad \text{for $j\ge 0$.}
\]
Let $\eta \sim \textup{Unif}\,(\partial \D)$ be independent of $\{\alpha_j,j\ge 0\}$, and set $\alpha_{-1}=-1$. 
In this probability space, we introduce the filtration
\begin{align*}
    \mathcal{F}_n = \sigma(\eta,\alpha_0,\cdots,\alpha_{n-1}),\qquad n\ge 0, 
\end{align*}
and define the CMV matrices
\begin{align*}
\cC_n:=\cC(\alpha_0,\cdots,\alpha_{n-2},\eta),\quad n\ge 1,\qquad \text{and}\qquad  \cC_{\infty} = \cC(\alpha_0,\alpha_1,\cdots).
\end{align*}
\end{assumption}

\subsection{Stein's method}
The Stein’s method is a collection of probabilistic techniques that estimate the distance between two probability distributions. In particular, it provides a useful characterization of the Gaussian distributions. Stein’s observation \cite{Stein72} was that a real-valued random variable $X$ has the standard normal $N(0,1)$ distribution if and only if, for every differentiable function $f:\R\mapsto \R$ such that the expectations $\ev [Xf(X)]$ and $\ev[ f'(X)]$ are finite and
\begin{align}\label{eq:Stein_1d}
    \ev[Xf(X)] = \ev[f'(X)].
\end{align}
From this observation, one expects that if $W$ is a random variable such that $\ev[Wf(W)-f'(W)]$ is close to zero for a large class of testing functions $f$, then the law of $W$ should be close to the Gaussian distribution.
% More precisely, for any bounded measurable function $g$, Stein proved that there exists a bounded solution to the differential equation
% \[
% f'(x) - xf(x) = g(x) -\ev[g(X)],\qquad X\sim N(0,1),
% \]
% hence $\ev[Wf(W)-f'(W)]=\ev[g(W)-g(X)]$. Then one may measure the distance between $W$ and $X$ by considering the the supreme of $\ev[g(X)-g(W)]$ over certain class of testing functions. 
We refer to the survey \cite{Chatterjee14} and the monograph \cite{NourdinPeccati} for a detailed treatment of Stein's method.

% Stein's proof in \cite{Stein72} used the method of exchangeable pairs. In \cite{Barbour90} Barbour introduced the diffusion approach to Stein's method. Goldstein and Rinott \cite{GoldsteinRinott_bias,GoldsteinReinert97} developed the method of size-biased and zero-bias couplings for Stein's method. Using Stein's method, Chen and Shao \cite{ChenShao04} obtained optimal Berry-Ess\'een type error bounds. For multivariate normal approximations the Stein's method was carried out by G\" otze \cite{gotze91}, Bolthausen and G\"otze \cite{BolthausenGotze}. Later  Chatterjee and Meckes \cite{CM} generalized the method of exchangeable pairs in the multivariate setting.
% We refer to the survey \cite{Chatterjee14} and the monograph \cite{NourdinPeccati} for a detailed treatment. 

The next result can be viewed as the multidimensional counterpart of Stein’s method.
\begin{proposition}[Lemma 4.1.3 in \cite{NourdinPeccati}]\label{prop:multi-Stein}
Let
$\Sigma = \{\Sigma(i,j): i,j=1,\ldots,d\}$
be a non-negative definite $d\times d$ matrix. Let $X=(X_1,\cdots,X_d)^\top$ be a random vector in $\mathbb{R}^d$. Then $X$ has the Gaussian $N_d(0,\Sigma)$ distribution if and only if
\begin{equation}
\ev\big[\langle X,\nabla f(X)\rangle_{\mathbb{R}^d}\big]
=
\mathbb{E}\big[\langle \Sigma,\operatorname{Hess} f(X)\rangle_{\mathrm{HS}}\big],
\end{equation}
for every $C^2$ function $f:\mathbb{R}^d\to\mathbb{R}$ having bounded first and second derivatives. Here, $\operatorname{Hess} f$ denotes the Hessian of $f$, i.e.~the $d\times d$ matrix whose entries are given by
$\bigl(\operatorname{Hess} f\bigr)_{i,j}
=\frac{\partial^2 f}{\partial x_i\,\partial x_j}$.
\end{proposition}

For our purpose, we use the following version of the multidimensional Stein's lemma. 
For a complex variable $z=x+iy$ the Wirtinger derivatives are defined as usual
\begin{align*} 
    \frac{\partial}{\partial z} = \frac12\left(\frac{\partial}{\partial x}-i\frac{\partial}{\partial y}\right),\qquad  \frac{\partial}{\partial \bar z} = \frac12\left(\frac{\partial}{\partial x}+i\frac{\partial}{\partial y}\right).
\end{align*} 
With a bit abuse of notations, we write $f\in C^1(\CC^k)$ if $\tl f(x_1,y_1,\cdots,x_k,y_k):=f(x_1+iy_1,\cdots,x_k+iy_k)$ has continuous partial derivatives, i.e.~$f$ is real differentiable, and we abbreviate $\partial_z = \frac{\partial}{\partial z}, \partial_{\bar z}=\frac{\partial}{\partial{\bar z}}$.
\begin{proposition}\label{prop:Gauss-Stein}
The vector ${\bf Z}=(Z_1,\dots,Z_k)^\top\in\CC^k$ has independent complex Gaussian entries $Z_j\sim \CC N(0,\sigma_j^2)$ for $1\le j\le k$ if and only if for any $F\in C^1(\CC^k)$ with bounded derivative we have 
\begin{align*}
    \ev[\ol Z_j F({\bf Z})] = \sigma_j^2\ev[\partial_{z_j}F({\bf Z})],\qquad 
     \ev[Z_j F( {\bf Z})] = \sigma_j^2 \ev[\partial_{\bar z_j}F({\bf Z})],\qquad 1\le j\le k.
\end{align*}
\end{proposition}
\begin{proof}
    By identifying $\CC^k \cong \R^{2k}$, the statement follows from Proposition \ref{prop:multi-Stein} and the identity
    $\partial_z\partial_{\bar z}= \partial_{\bar z}\partial_z =\tfrac14\Delta$. 
\end{proof}

The following results can be viewed as the analogue of Stein's method for the Verblunsky coefficients corresponding to the circular beta-ensemble. The proof is a straightforward application of integration by parts, its proof is postponed to Section \ref{sec:Appendix} (the Appendix). 
\begin{proposition}\label{prop:beta-Stein}
    Suppose $\alpha\in\D$ is a rotationally invariant random variable such that $|\alpha|^2\sim\textup{Beta}(1,a)$, then for $f,g\in C^1(\overline{\D})$ we have 
    \begin{align}
       a\ev[\alpha (f\cdot g)(\alpha)]&= \ev[\rho^2(\partial_{\bar z}f\cdot g)(\alpha)] + \ev[\rho^2 (f\cdot \partial_{\bar z} g)(\alpha)]\label{eq:beta-Stein1}\\
       a\ev[\bar \alpha (f\cdot g)(\alpha)] &=  \ev[\rho^2(\partial_{ z}f\cdot g)(\alpha)] + \ev[\rho^2 (f\cdot \partial_{ z} g)(\alpha)]\label{eq:beta-Stein2}
    \end{align}
\end{proposition}

In the case when $g=1$, we get the following simplified identities
\begin{align}\label{eq:beta-Stein_simple}
    a\ev[\bar \alpha_j f(\alpha_j)] =\ev\left[\rho^2\cdot \partial_{z}f(\alpha)\right],\qquad     a\ev[\alpha f(\bar\alpha)] =\ev\left[\rho^2\cdot \partial_{\bar z}f(\alpha_j)\right].
\end{align}

\subsection{Ablowitz-Ladik system}
In \cite{AL1,AL2} Ablowitz and Ladik introduced the following system of nonlinear equations
\begin{equation}\label{eq:AL}
\begin{split}
        i\frac{d}{dt}\alpha_j&= - (\alpha_{j+1}+\alpha_{j-1}-2\alpha_j) + |\alpha_j|^2(\alpha_{j-1}+\alpha_{j+1})\qquad \text{for $\alpha_j\in \D$,}
    \end{split}
\end{equation}
as the integrable discretizations of the nonlinear Schr\"odinger equation 
\[
i \partial_t \psi = -\partial_x^2 \psi + 2|\psi|^2\psi.
\]
We refer to the system of equations \eqref{eq:AL} as the Ablowitz-Ladik (AL) system. Under the mapping $\alpha_j(t)\mapsto e^{-2it}\alpha_j(t)$, one obtains the system of equations
\[
 i\frac{d}{dt}\alpha_j = -\rho_j^2(\alpha_{j-1}+\alpha_{j+1}),\qquad \rho_j^2:=1-|\alpha_j|^2.
\]
which we call the reduced Ablowitz-Ladik system, see \cite{GravaMazzuca}.

For the finite or periodic reduced AL system, we introduce the Hamiltonian 
\begin{align*}
    H_n = -\sum_{j=0}^{n-1}(\alpha_j\bar\alpha_{j+1} +\bar \alpha_j\alpha_{j+1}),
\end{align*}
and the Poisson bracket
\begin{equation}\label{eq:AL-bracket}
 \{f,g\}_n
 :=i\sum_{j=-1}^{n-1}\rho_j^2
 \left(
   \partial_{\ol\alpha_j}f\,\partial_{\alpha_j}g
  -\partial_{\alpha_j}f\,\partial_{\ol\alpha_j}g
 \right),
 \qquad \rho_j^2:=1-|\alpha_j|^2.
\end{equation}
Then it is direct to check that 
\[
\frac{d}{dt}\alpha_j = \{\alpha_j,H_n\}_n = i\rho_j^2(\alpha_{j-1}+\alpha_{j+1}).
\]
Let us also note that the related Hamiltonian flow generated by the imaginary part of $\sum_{j=0}^{n-1} \bar \alpha_j\alpha_{j+1}$ is known as the Schur flow, see \cite{AL2,SimonSchur}.

The Poisson bracket for the AL system will play an important role in identifying the distributions of the limits in Theorem \ref{thm:main}.  We record below the related identities derived by Cantero and Simon \cite{CanteroSimon} (see also \cite{Nenciu_note} for further results on the Poisson brackets for OPUC).
We provide a sketch of the proof in Section \ref{sec:Appendix} (the Appendix) for completeness.
\begin{proposition}[\cite{CanteroSimon}]\label{prop:Poisson_id}
    Let $\mu$ be a probability measure on the circle supported on $N+1$ points, then for $n\le N$ the Poisson bracket of the monic orthogonal polynomials associated with $\mu$ satisfy
    \begin{align}
        \{\Phi_n(z),\Phi_n(w)\}_n &=  \{\Phi_n^*(z),\Phi_n^*(w)\}_n=0,\label{id:Poisson1}\\
        \{\Phi_n(z),\Phi_n^*(w)\}_n &= \frac{iw}{z-w}(\Phi_n(z)\Phi_n^*(w)-\Phi_n(w)\Phi_n^*(z))\label{id:Poisson2}.
    \end{align}
    The $z=w$ case is extended by continuity.
\end{proposition}

We conclude this section with a brief discussion on the possibility of extending the Poisson bracket to infinitely many Verblunsky coefficients. In general, this could be quite challenging. To ensure the infinite series 
\begin{align*}
    \{f,g\}_\infty:= i\sum_{j\ge -1}\rho_j^2(\partial_{\bar \alpha_j}f \cdot \partial_{\alpha_j}g - \partial_{\alpha_j}f\cdot\partial_{\bar \alpha_j}g),\qquad \rho_j^2=1-|\alpha_j|^2.
\end{align*}
to be well-defined, a natural sufficient condition is to require that both $f$ and $g$ have square summable derivatives:
\[
f,g\in \mathcal{H}=\{h\in C^1(\CC^\infty): \sum_{j\ge -1} |\partial_{\alpha_j} h|^2 + |\partial_{\bar \alpha_j}h|^2 <\infty\}
\]
However, this condition is too restrictive for our purpose,  as even the first trace $T_1:=-\sum_{j\ge -1}\alpha_j\bar \alpha_{j+1}\notin \mathcal{H}$.

Moreover, uniform convergence of the functions alone is not sufficient to ensure the convergence of their infinite Poisson brackets. For example, consider $f_n=\bar \alpha_{n},g_n=\alpha_{n}$. Then for any sequence $\alpha_n\to 0$ we have $f_n,g_n\to 0$, while
for $N\ge n$ (including $N=\infty$)
\[
\{f_n,g_n\}_N = i\rho_n^2 \to i.
\]
This simple example illustrates that, to make sense of $\{\cdot,\cdot\}_\infty$, one also needs uniform control of the tails of the Poisson brackets, in addition to the convergence of the test functions. See Remark \ref{rmk:infinite} below for further discussion.

\section{A direct proof of the identity \eqref{eq:1st}}\label{sec:1st}
To illustrate the application of Stein's method to the traces of circular beta-ensemble, we first give a direct proof of \eqref{eq:1st}. The Gaussian limits for the traces of higher powers of circular beta-ensemble will be proved in a similar manner, using several key identities involving the Ablowitz-Ladik Poisson bracket.

% By Proposition \ref{prop:KN2}, we have 
% \begin{align*}
%     T_k^{(n)}\ed \sum_{j=1}^n e^{i k\theta_j},\qquad \{e^{i\theta_1},\dots,e^{i\theta_n}\}\ed \text{C$\beta$E}_n.
% \end{align*}
\begin{proof}[A direct proof of \eqref{eq:1st}]
    Recall Assumption \ref{coupling}. We denote the traces of $\cC_\infty$ by $T_k:=\tr(\cC_\infty^k)$ if the limit exists. In particular, we have  
    \[
    T_1=-\lim_{n\to\infty}\sum_{j=-1}^n\alpha_j\bar\alpha_{j+1}=-\sum_{j= -1}^\infty  \alpha_j\bar \alpha_{j+1}
    \]
    with $\alpha_{-1}=-1$. Note that the limit exists since $\ev[|\alpha_n\bar \alpha_{n+1}|^2]=O(n^{-2})$. By Proposition \ref{prop:Gauss-Stein}, it suffices to prove that for any $f\in C^1(\CC)$ with bounded first derivative, we have 
    \begin{align}\label{eq:1st_goal}
        \ev[\ol T_1 f(T_1)]=\tfrac{2}{\beta} \ev[\partial_{z}f(T_1)],\qquad  \ev[ T_1 f( T_1)]=\tfrac{2}{\beta} \ev[\partial_{\bar z}f( T_1)].
    \end{align}
    Note that for $j\ge 0$ we have 
    \begin{align*}
    \partial_{\alpha_j} f(T_1)& = -\bar \alpha_{j+1} \partial_zf(T_1)-\bar\alpha_{j-1}\partial_{\bar z}f(T_1),\\
    \partial_{\bar \alpha_j}f(T_1)& = -\alpha_{j-1}\partial_zf(T_1)-\alpha_{j+1}\partial_{\bar z}f(T_1).
\end{align*}
Recall $\beta_j=\frac{\beta}{2}(j+1)$.
By the Stein's identities \eqref{eq:beta-Stein_simple}, we have 
\begin{align}
    &\tfrac{\beta}{2}\ev[\bar\alpha_j \alpha_{j+1}f(T_1)]\notag\\
    &\quad = \ev[\beta_{j+1}\bar\alpha_j \alpha_{j+1}f(T_1)]-\ev[\beta_j\bar\alpha_j\alpha_{j+1}f(T_1)]\notag\\
    &\quad=\ev[\rho_{j+1}^2\bar \alpha_j \partial_{\bar \alpha_{j+1}}f(T_1)]-\ev[\rho_j^2\alpha_{j+1}\partial_{\alpha_j}f(T_1)]\notag\\
    % &=\ev[-\rho_{j+1}^2|\alpha_j |^2f'(T_1)] -\ev[-\rho_j^2|\alpha_{j+1}|^2 f'(T_1)]\\
&\quad=\ev\left[\partial_zf(T_1)\big(|\alpha_{j+1}|^2\rho_j^2-|\alpha_j|^2\rho_{j+1}^2)\right]+\ev\left[\partial_{\bar z}f(T_1)\big(\rho_j^2\bar\alpha_{j-1}\alpha_{j+1}-\rho_{j+1}^2\bar \alpha_j\alpha_{j+2}\big)\right]\label{eq:telescoping}.
\end{align}
Note that the terms on the right hand side of \eqref{eq:telescoping} are telescoping series.
Together with
\[
\tfrac{\beta}{2}\ev[\alpha_0f(T_1)] =\ev[\rho_0^2\partial_{\bar \alpha_0}f(T_1)]=\ev[\rho_0^2\partial_{z}f(T_1)-\rho_0^2\alpha_1\partial_{\bar z}f(T_1)],
\]
and the identity
\[|\alpha_{j+1}|^2\rho_j^2-|\alpha_j|^2\rho_{j+1}^2 = |\alpha_{j+1}|^2 - |\alpha_j|^2,
\]
we get that 
\begin{align*}
    \tfrac{\beta}{2}\ev[\ol T_1 f(T_1)]&=\tfrac{\beta}{2}\ev[\alpha_0f(T_1)]-\sum_{j=0}^{\infty} \tfrac{\beta}{2}\ev[\bar \alpha_j\alpha_{j+1}f(T_1)] \\
    &= \sum_{j=-1}^\infty \ev[\partial_zf(T_1)( |\alpha_j|^2-|\alpha_{j+1}|^2)] = \ev[\partial_zf(T_1)]
\end{align*}
as claimed. This proves the first equality of \eqref{eq:1st_goal}, the second one can be proved similarly by applying Stein's method to the conjugated terms $\alpha_j\bar \alpha_{j+1}$ for $j\ge -1$. By Proposition \ref{prop:Gauss-Stein}, we conclude that $T_1\sim \CC N(0,\frac{2}{\beta})$.
\end{proof}

\section{Proofs of Theorems \ref{thm:main} and \ref{thm:1st}}\label{sec:proof}
Recall the coupling in Assumption \ref{coupling}, and the CMV matrices $\cC_n,n\ge 1$ and $\cC_\infty$.
Introduce the auxiliary sequence
\[
\cC_{\infty}^{(n)}=\cC(\alpha_0,\dots,\alpha_{n-1}),\qquad  \mathtt{T}_{k}^{(n)}:=\tr\big((\cC_\infty^{(n)})^k\big).
\]
It is direct to check that $\cC_\infty^{(n)}$ is exactly the $n\times n$ upper-left corner of $\cC_\infty$. We outline the proof ideas of Theorems \ref{thm:main} and \ref{thm:1st}:
\begin{enumerate}
    \item For every fixed $k$, we have $|T_k^{(n)}-\mathtt{T}_k^{(n)}|\to 0$ and $\mathtt{T}_k^{(n)}\to T_k$ a.s.~and in $L^2$ as $n\to\infty$, where $T_k=\tr(\cC_\infty^k)$.
    \item Identifying the distribution of $T_k$ using Stein's method. This step also uses the AL Poisson bracket identities for the pre-limiting object $\mathtt{T}_k^{(n)}$.
\end{enumerate}
Note that the first step should be known to experts in the field, see e.g.~\cite{CN_GMC}, we include the proofs in Section \ref{subsec:reduction} for completeness. In Section \ref{subsec:proof}, we identify the distribution of $T_k$ and complete the proofs of Theorems \ref{thm:main} and \ref{thm:1st}. 

\subsection{Replacing $T_k^{(n)}$ by $\mathtt{T}_k^{(n)}$, and convergence of $\mathtt{T}_k^{(n)}$}\label{subsec:reduction}
Let $\Phi_n,\Phi_n^*$ be the sequence of monic and reversed orthogonal polynomials with respect to the sequence of Verblunsky coefficients $\{\alpha_j,j\ge 0\}$. By Proposition \ref{prop:CMV1}, we have 
\begin{align*}
    \Phi_n(z)=\det(z\mathbf{I}_n-\cC_\infty^{(n)}),\quad \Phi_n^*(z) = z^n\ol{\Phi_n(1/\bar z)} = \det(\mathbf{I}_n-z\ol{\cC}_\infty^{(n)}).
\end{align*}
% Let $\Phi_k^*,k\ge 0$ be the sequence of reversed polynomials corresponding to the infinite sequence of Verblunsky coefficients $\alpha_0,\alpha_1,\dots$, i.e.~they satisfy the recursion 
By the Szeg\H o recursion \eqref{eq:Szego1} and the fact that the zeros of $\Phi_k^*$ lie outside of the unit disk, we have 
% \begin{align*}
%     \Phi_{k+1}^*(z)  = \Phi_k^*(z) -\alpha_k z\Phi_k(z).
% \end{align*}
% The zeros of $\Phi_k^*$ lie outside of the unit disk. Taking logarithm, we observe that 
\begin{align}\label{eq:logPhi*}
    \log \Phi_{k+1}^*(z) = \log \Phi_k^*(z) +\log \left(1-\alpha_k zB_k(z)\right),\qquad B_k(z):=\frac{\Phi_k(z)}{\Phi_k^*(z)},
\end{align}
where the Blaschke product $B_k(z)$ is analytic satisfying $|B_k(z)|<1$ for $z\in\D$ and $k\ge 1$, and $B_0\equiv 1$. 
Under the coupling described in Assumption \ref{coupling}, we define
\begin{align}\label{eq:characteristic}
    X_n(z):= \det(\mathbf{I}_n-z\ol{\cC_n})=\Phi_{n-1}^*(z) - \eta z \Phi_{n-1}(z),
\end{align}
then by Proposition \ref{prop:KN2}, the zeros of $X_n$ have the same distribution as the size $n$ circular beta-ensemble, i.e., we have
\begin{align*}
    X_n(z) = \prod_{j=1}^n \left(1- z e^{-i\theta_j}\right),\qquad \{e^{i\theta_j},1\le j\le n\}\sim \text{C$\beta$E}_n.
\end{align*}

Consider the expansion of $\log X_n$ and $\log \Phi_n^*$ around $0$, their Taylor coefficients are related to the traces of $\cC_n$ and $\cC_\infty^{(n)}$ as follows
\begin{align}\label{eq:logPhi*_Taylor}
    \log \Phi_n^*(z) =-\sum_{k\ge 1}\frac{1}{k}\ol{\mathtt{T}}_{k}^{(n)}z^k:= \sum_{k\ge 1}\mathtt{t}_{k}^{(n)} z^k,\quad  \log X_n(z)=-\sum_{k\ge 1}\frac{1}{k}\ol T_k^{(n)}z^k := \sum_{k\ge 1} t_k^{(n)}z^k.
\end{align}
Write $B_n(z) := \sum_{k\ge 0}b_k^{(n)}z^k$, by \eqref{eq:characteristic} we have 
\begin{align*}
    \log X_n(z) = \log \Phi_{n-1}^*(z) + \log (1-\eta z B_{n-1}(z)).
\end{align*}
By the general OPUC theory on the ratio estimates (Theorem 1.7.4 in \cite{OPUC1}), since $\alpha_n\to 0$ a.s.~as $n\to\infty$ (see Proposition \ref{prop:ab} below), we have 
\[
B_n(z) \to 0 \qquad \text{a.s. as $n\to\infty$, uniformly on the compact sets of $\D$}.
\]
This implies $b_k^{(n)}\to 0$ a.s.~as $n\to\infty$ for any fixed $k$. Since $|\eta|=1$ is independent of $\alpha_j,j\ge 0$, we can conclude that $|\log X_n-\log \Phi_{n-1}^*|\to 0$ a.s.~as $n\to\infty$, uniformly on the compact sets of $\D$. This in turn implies that 
\begin{align}\label{eq:t_diff}
    t_{k}^{(n)} - \mathtt{t}_k^{(n-1)} =[z^k]\log(1-\eta zB_{n-1}(z))
    % = - \sum_{j=1}^k \frac{\eta^j}{j}[z^{k-j}] B_{n-1}(z)^j,
\end{align}
vanishes a.s.~as $n\to\infty$. Here we adopt the notation $[z^k]h(z)$ as the coefficient of $z^k$ in the power series expansion of $h(z)$ around the point $z = 0$.

Note that the arguments above were already used in \cite{CN_GMC} to prove the convergence and identify the limits of $\log \Phi_n^*$ and $\log X_n$, see e.g.~Proposition 3.1 in \cite{CN_GMC}. Their proof relies on the CLTs of Jiang and Matsumoto (Theorem \ref{thm:clt1}) for the traces of circular beta-ensemble. Since our aim is to provide a direct proof of Theorem \ref{thm:main}, rather than invoking the results of \cite{CN_GMC} on the convergence, we give a self-contained estimate on $b_k^{(n)}$ and $t_k^{(n)}-\mathtt{t}_k^{(n)}$ below.
\begin{proposition}\label{prop:ab}
We have $\alpha_n\to 0$ a.s.~and in $L^2$ as $n\to\infty$.
For any fixed $k$, we have $b_k^{(n)}\to 0$ a.s.~as $n\to\infty$. Moreover, we have 
\[
 \ev[(b_k^{(n)})^2]\le C_{\beta}\frac{k^2}{n-k}, \qquad k\le n-1,
\]
where $C_\beta$ is independent of $k$ and $n$. 
Consequently, $T_k^{(n)}-\mathtt{T}^{(n)}_k\to 0$ a.s.~and in $L^2$ as $n\to\infty$. 
%     \begin{align*}
%         b_k^{(n)}\to 0\quad \text{a.s. and in $L^2$ as $n\to\infty$,}
%     \end{align*}
% and $ \ev[(b_k^{(n)})^2]\le C_{\beta}\frac{k^2}{n-k}$. 
\end{proposition}
\begin{proof}
    By definition, we have
    \begin{align*}
        \pr(|\alpha_n|^2\ge x) = (1-x)^{\frac{\beta}{2}(n+1)},
    \end{align*}
    which is summable in $n$ for $x>0$. This shows that $\alpha_n\to 0$ a.s.~by the Borel-Cantelli Lemma. The $L^2$ convergence follows from the exact computation
    \begin{align}\label{eq:alpha_moments}
        \ev[|\alpha_n|^2]= \frac{1}{1+\frac{\beta}{2}(j+1)}.
    \end{align}
    (Indeed, we have $\alpha_n\to 0$ in $L^p$ for every $0<p<\infty$.)

    We now turn to the estimate of $b_k^{(n)}$.  By the Szeg\H o recursion \eqref{eq:Szego1}, we have 
    \begin{align*}
        B_{n+1}(z)= zB_n(z)-\bar\alpha_n +\alpha_n zB_n(z)B_{n+1}(z).
    \end{align*}
    By comparing the Taylor coefficients, we have 
    \begin{align*}
        b_k^{(n+1)} = b_{k-1}^{(n)}+\alpha_n [z^{k-1}](B_nB_{n+1}).
    \end{align*}
    Using the fact that $|(B_nB_{n+1})(z)|<1$ for $z\in\D$, and Cauchy's formula, we have $|[z^{k-1}](B_nB_{n+1})|\le 1$. Iterating this argument leads to the (rough) bound 
    \begin{align*}
        |b_k^{(n+1)}|\le |b_{k-1}^{(n)}|+|\alpha_n|\le |b_{k-2}^{(n-1)}|+|\alpha_n|+|\alpha_{n-1}|\le  \cdots\le \sum_{j=n-k}^{n}|\alpha_j|.
    \end{align*}
    This shows that $b_k^{(n)}\to 0$ a.s.~as $n\to\infty$ for every fixed $k$.
    By Minkowski's inequality and \eqref{eq:alpha_moments}, we have
    \begin{align*}
        \|b_k^{(n)}\|_2\le \sum_{j=n-k-1}^{n-1}\|\alpha_j\|_2\le \frac{k+1}{\sqrt{1+\frac{\beta}{2}(n-k)}}.
    \end{align*}
    This proves the $L^2$ convergence of $b_k^{(n)}$.
    Recall the definitions $T_k^{(n)}=-k\bar t_{k}^{(n)}$ and $\mathtt{T}_k^{(n)}=-k\ol{\mathtt{t}}_k^{(n)}$, the last statement follows from \eqref{eq:t_diff} and the convergence of $b_k^{(n)}$. 
    % we have
    % \begin{align*}
    %     B_{n+1}(z) = \frac{zB_n(z)-\bar \alpha_n}{1-\alpha_n zB_n(z)} = -\bar\alpha_n + \rho_n^2 \sum_{j\ge 0}\alpha_n^j (zB_n(z))^{j+1},
    % \end{align*}
    % hence $b_0^{(n+1)} = -\bar \alpha_n$ and 
    % \begin{align*}
    %     b_k^{(n+1)} = \rho_n^2 \sum_{j=0}^{k-1}\alpha_n^j [z^{k-j-1}]B_n(z)^{j+1},\qquad k\ge 1.
    % \end{align*}
    % Since $\|zB_n\|_{L^\infty}\le 1$ and $B_n$ is analytic in $\D$, by Cauchy's formula we have $|[z^{k-j-1}]B_n(z)^{j+1}|\le 1$ for any $j\le k-1$. This implies the (rough) bound
    % \begin{align*}
    %     b_k^{(n+1)}\le \sum_{j=0}^{k-1}\alpha_n
    % \end{align*}
\end{proof}
Let us note that the convergence in Proposition \ref{prop:ab} also holds in every $L^p$ space for $0< p<\infty$, we restrict ourselves to the $L^2$-convergence for simplicity. 
\begin{remark}\label{rmk:LLN}
In forthcoming joint work with Liu and Najnudel \cite{LLN}, we can actually show that $\sqrt{n}(b_1^{(n)},\cdots,b_k^{(n)})\Rightarrow (\xi_1,\cdots,\xi_k)$, where $\xi_i,1\le i\le k$ are i.i.d.~standard complex normal variables for any fixed $k$. This is a key ingredient in our proof of  
\[
\sqrt{n} B_n(z) \Rightarrow \sqrt{\tfrac{2}{\beta}}\textup{GAF}(z),\qquad \textup{GAF}(z):=\sum_{j\ge 0}\xi_j z^j
\]
in law, uniformly on compact sets of $\D$, where we identify the limit as a constant ($\sqrt{2/\beta}$) multiple of the so-called Gaussian analytic function. See \cite{PV} for the study of the zeros of the GAF function.
\end{remark}

The next result establishes the convergence of $\mathtt{t}_k^{(n)}$ using a martingale argument and the estimates in Proposition \ref{prop:ab}. 
\begin{proposition}\label{prop:mg}
For any fixed $k$, there exists $t_k$ such that 
    \begin{align}
        \mathtt{t}_k^{(n)} \to t_k\quad \text{almost surely, and in $L^2$ as $n\to\infty$}.
    \end{align}
Moreover, we have $t_k=-\frac{1}{k}\ol T_k=-\frac{1}{k}\tr(\ol{\cC_\infty^k})$.
\end{proposition}
\begin{proof}
    Recall the recursion \eqref{eq:logPhi*}. Since $B_n$ is $\mathcal{F}_n$ measurable, $\alpha_n$ is rotationally invariant and independent of $\mathcal{F}_n$, we conclude that $\log \Phi_n^*(z),n\ge 0$ is a martingale for any fixed $z\in\D$. By  the same argument, one can show that for any fixed $k$, $\mathtt{t}_k^{(n)},n\ge 1$ is a $\mathcal{F}$-martingale. 
    (Note that by a more careful estimate of exponential moments of $\log\Phi_n^*$, one can show that $\log\Phi_n^*(z)\to\log\Phi_\infty^*(z)$  a.s.~as $n\to\infty$, uniformly on compacts sets of $\D$, see Proposition 3.1 in \cite{CN_GMC} for more details.) 

    For our purposes, we will derive the convergence using Proposition \ref{prop:ab}. Fix $k\ge 1$, we define the martingale difference
    \begin{align*}
        m_k^{(n)}:=\mathtt{t}_{k}^{(n+1)}-\mathtt{t}_k^{(n)}=-\sum_{j=1}^k \frac{\alpha_n^j}{j}[z^{k-j}]B_n(z)^j.
    \end{align*}
    Using the independence and rotational invariance of $\alpha_n$, we have 
    \begin{align*}
        \ev[|m_k^{(n)}|^2] = \sum_{j=1}^k\frac{\ev[|\alpha_n|^{2j}]}{j^2}\ev\left[\big|[z^{k-j}]B_n(z)^j\big|^2\right]\le C_\beta\frac{k^2}{n^2}.
    \end{align*}
    Indeed, when $k=1$, we have $\ev[|\alpha_n|^2]=O(n^{-1})$ and $\ev[|b_{0}^{(n)}|^2]=\ev[|\alpha_{n-1}|^2]=O(n^{-1})$. For $k\ge 2$, we simply use the bounds $\ev[|\alpha_n|^{2j}]=O(n^{-j})$ and $\ev\left[\big|[z^{k-j}]B_n(z)^j\big|^2\right]\le 1$. Therefore, $\sum_{n}\ev[|m_k^{(n)}|^2] <\infty$ and we have $\mathtt{t}_{k}^{(n)}$ is bounded in $L^2$. This implies that there exists $t_k$ such that $\mathtt{t}_k^{(n)}\to t_k$ and $\mathtt{T}_k^{(n)}=-k\bar {\mathtt{t}}_k^{(n)}\to -k\bar t_k$ a.s~and in $L^2$ as $n\to\infty$. 
    
    It remains to prove that $-k\bar t_k = T_k=\tr(\cC_\infty^k)$. Define $S_k^{(n)}:=\sum_{j=1}^n [\cC_\infty^k]_{jj}$. Since $\cC_\infty$ is five-diagonal, $S_k^{(n)}$ is well defined, and we have $T_k=\lim_{n\to\infty} S_k^{(n)}$ whenever the limit exists. Therefore, it suffices to $S_k^{(n)}-\mathtt{T}_k^{(n)}\to 0$ a.s.~and in $L^2$ as $n\to\infty$ for any fixed $k$. The CMV matrices $\cC_n$ and $\cC_\infty^{(n)}$ are five-diagonal, and only differ by the last Verblunsky coefficient. Therefore, we have $[(\cC_\infty^{(n)})^k]_{jj}=[\cC_n^k]_{jj}$ for $j\le n-k$. Every entry of $\cC_\infty,\cC_n$ has modulus of at most one, and there are at most $k5^k$ paths that contains the random variable $\eta$ in $\cC_n$. Thus, we get the following crude bound
    \[
    |S_k^{(n)}-\mathtt{T}_k^{(n)}| \le k5^k\max_{|j-n|\le  k}|\alpha_j|\le  k5^k\sum_{j=n-k}^{n+k}|\alpha_j|.
    \]
   Applying Proposition \ref{prop:ab} shows that the error term $ |S_k^{(n)}-\mathtt{T}_k^{(n)}| \to 0$ a.s.~and in $L^2$ as $n\to\infty$. This finishes the proof. 
\end{proof}

\subsection{Identifying the distributions of $T_k$}\label{subsec:proof}
The goal of this section is to complete the proof of Theorems \ref{thm:main} and \ref{thm:1st}. We first prove Theorem \ref{thm:1st} using the convergence results established in Section \ref{subsec:reduction} and the proof of \eqref{eq:1st} in Section \ref{sec:1st}.
\begin{proof}[Proof of Theorem \ref{thm:1st}]
  Propositions \ref{prop:ab} and \ref{prop:mg} prove that for any fixed $k$, we have 
\[
(T_1^{(n)},\cdots, T_k^{(n)})\to (T_1,\cdots,T_k)\qquad \text{a.s.~and in $L^2$ as $n\to\infty$,}
\]
where $T_i=\tr(\cC_\infty^i)$. Together with the distributional identity \eqref{eq:1st} proved in Section \ref{sec:1st} completes the proof of the statement. 
\end{proof}

In order to complete the proof of Theorem \ref{thm:main}, it remains to identify the joint distribution of $T_k$ for general $k$. Recall that $\mathtt{T}_k^{(n)}$ and $T_k$ are defined as traces of $(\cC_\infty^{(n)})^k$ and $\cC_\infty^k$, respectively. These traces can be related to the Taylor coefficient of $\log\Phi_n^*$ and its limit. We first compute the Ablowitz-Ladik Poisson bracket of $\log \Phi_k^*$ using  Proposition \ref{prop:Poisson_id}.
\begin{proposition}\label{prop:logAL}
    For any $N\ge n$ and $z,w\in\D$ we have 
    \begin{align}
        \{\log \Phi_n^*(z),\log \Phi_n^*(w)\}_N &= 0,\label{id:logPoisson1}\\
        \{\log \overline{\Phi_n^*}(z),\log \Phi_n^*(w)\}_N &= i\frac{\bar zw}{1-\bar z w}(1-\ol{B_n}(z)B_n(w)),\label{id:logPoisson2}.
    \end{align}
    For $i,j\ge 1$ and $N\ge n$ we have
    \begin{align}
        \{\mathtt{t}_i^{(n)},\mathtt{t}_j^{(n)}\}_N=0,\qquad \{\ol{\mathtt{t}}_i^{(n)},\mathtt{t}_j^{(n)}\}_N=i\left(\delta_{ij}-\sum_{k=1}^{\min\{i,j\}}\bar b_{i-k}^{(n)}b_{j-k}^{(n)}\right)\label{id:logPoisson3}.
    \end{align}
\end{proposition}
\begin{proof}
    By \eqref{id:Poisson1} we have 
    \begin{align*}
        \{\log\Phi_n^*(z),\log\Phi_n^*(w)\}_N = \frac{1}{\Phi_n^*(z)\Phi_n^*(w)}\{\Phi_n^*(z),\Phi_n^*(w)\}_N =0,
    \end{align*}
    which proves \eqref{id:logPoisson1}.
    To prove \eqref{id:logPoisson2}, we note that  
    \[
    \ol{\Phi_n^*}(z)=\bar z^n\Phi_n(1/\bar z),\qquad \frac{\Phi_n^*(1/\bar z)}{\Phi_n(1/\bar z)} = \frac{\ol{\Phi_n}(z)}{\ol{\Phi_n^*}(z)}=\ol{B_n}(z).
    \]
    Thus, by \eqref{id:Poisson2} we have  
    \begin{align*}
        \{\log\ol{\Phi_n^*}(z),\log\Phi_n^*(w)\}_N &= \frac{1}{\ol{\Phi_n^*}(z)\Phi_n^*(w)}\{\ol{\Phi_n^*}(z),\Phi_n^*(w)\}_N\\
        &=\frac{\bar z^n}{\ol{\Phi_n^*}(z)\Phi_n^*(w)}\left(\frac{iw}{\frac{1}{\bar z}-w}\big(\Phi_n(1/\bar z)\Phi_n^*(w)-\Phi_n(w)\Phi_n^*(1/\bar z)\big)\right)\\
        &=\frac{iw\bar z}{1-w\bar z}\left(1-\ol{B_n}(z)B_n(w)\right),
    \end{align*}
    where  in the last step we used 
    \[
    \frac{\bar z^n\Phi_n^*(1/\bar z)}{\ol{\Phi_n^*}(z)}= \frac{\ol{\Phi_n}(z)}{\ol{\Phi_n^*}(z)}=\ol{B_n}(z).
    \]
   Finally, since
    \begin{align*}
        \frac{\bar zw}{1-\bar zw} = \sum_{j\ge 1}(\bar zw)^j,\qquad \ol{B_n}(z)B_n(w) = \sum_{k,l\ge 0} \bar b_{k}^{(n)}b_l^{(n)} \bar z^k w^l,
    \end{align*}
     comparing the Taylor coefficients of the left and right hand sides of \eqref{id:logPoisson1} and \eqref{id:logPoisson2} yields \eqref{id:logPoisson3}. This completes the proof. 
\end{proof}

We are now ready to prove Theorem \ref{thm:main}.
\begin{proof}[Proof of Theorem \ref{thm:main}]
Recall that $|\alpha_j|^2\sim \textup{Beta}(1,\beta_j)$ with $\beta_j=\frac{\beta}{2}(j+1)$, and recall that the Poisson bracket is defined as
\begin{align*}
    \{F,G\}_{n} = i\sum_{j=-1}^{n-1}\rho_j^2(\partial_{\bar \alpha_j}F\partial_{ \alpha_j}G-\partial_{ \alpha_j}F\partial_{\bar \alpha_j}G).
\end{align*}
Taking expectation for each term and applying Proposition \ref{prop:beta-Stein}, we have 
\begin{align*}
    \ev[\rho_j^2\partial_{\bar \alpha_j}F\partial_{ \alpha_j}G] &=  \beta_j \ev[\alpha_j F\partial_{\alpha_j}G]-\ev[\rho_j^2 F\partial_{\bar \alpha_j}\partial_{\alpha_j}G],\\
    \ev[\rho_j^2\partial_{ \alpha_j}F\partial_{\bar \alpha_j}G]&=  \beta_j \ev[\bar \alpha_j F\partial_{\bar \alpha_j}G]-\ev[\rho_j^2 F\partial_{ \alpha_j}\partial_{\bar \alpha_j}G],
\end{align*}
wheneve the Wirtinger derivatives exist and have finite expectation.
Taking the difference yields that 
\begin{align*}
    \ev[\rho_j^2(\partial_{\bar \alpha_j}F\partial_{ \alpha_j}G-\partial_{ \alpha_j}F\partial_{\bar \alpha_j}G)] = \beta_j \ev[F(\alpha_j\partial_{\alpha_j}-\bar\alpha_j\partial_{\bar \alpha_j})G].
\end{align*}
Summing over $j$, we obtain that 
\begin{align}\label{eq:Poisson-D}
    \ev[\{F,G\}_n] = i \frac{\beta}{2}\ev[F\mathcal{D}_n G],
\end{align}
where we introduce the differential operator  
\begin{align*}
    \mathcal{D}_n = \sum_{j= -1}^{n-1}(j+1)(\alpha_j\partial_{\alpha_j}-\bar \alpha_j\partial_{\bar \alpha_j}).
\end{align*}
We claim that the traces $\mathtt{T}_k^{(n)}$ can be viewed as the 'eigenvectors' of $\cD_n$. The proof of the claim is postponed after the proof of the theorem.
\begin{claim}\label{claim:D}
For any fixed $k$, we have $\mathcal{D}_n\ol{\mathtt{T}}_k^{(n)} = k\ol{\mathtt{T}}_k^{(n)}$.
    % \begin{align*}
    %     \mathcal{D}_n\ol{\mathtt{T}}_k^{(n)} = k\ol{\mathtt{T}}_k^{(n)},\qquad \mathcal{D}_\infty \ol{T}_k = k \ol{T}_k.
    % \end{align*}
\end{claim}
Let ${\mathbf T}^{(n)}=(\mathtt{T}_1^{(n)},\mathtt{T}_2^{(n)},\cdots,\mathtt{T}_k^{(n)})$, $F\in C^1(\CC^k)$ with bounded Wirtinger derivatives, and choose 
$G(\mathbf{T}^{(n)})=\ol{\mathtt{T}}_j^{(n)}$. By \eqref{eq:Poisson-D} and Claim \ref{claim:D}, we have  
    \begin{align}\label{eq:remove_Poisson}
        \ev\left[\{F(\mathbf{T}^{(n)}),\overline{\mathtt{T}}_j^{(n)}\}_n\right] &= i\frac{\beta}{2}\ev\left[F(\mathbf{T}^{(n)})\mathcal{D}_n\overline{\mathtt{T}}_j^{(n)}\, \right]=i\frac{\beta j}{2}\ev\left[F(\mathbf{T}^{(n)})\overline{\mathtt{T}}_j^{(n)}\,\right].
    \end{align}
    On the other hand, since $\mathtt{T}_j^{(n)} = -j\ol{\mathtt{t}}_j^{(n)}$, it follows from \eqref{id:logPoisson3} that 
    \begin{equation}\label{eq:remove_Poisson2}
    \{F(\mathbf{T}^{(n)}),\ol {\mathtt{T}}_j^{(n)}\}_n = i\sum_{m=1}^k mj\left(\delta_{mj}-\sum_{l=1}^{\min\{m,j\}}\bar b_{m-l}^{(n)}b_{j-l}^{(n)}\right)\partial_{ z_m}F(\mathbf{T}^{(n)}).
    \end{equation}
    By Propositions \ref{prop:ab} and \ref{prop:mg}, we have 
    \[
    \bar b_{m-l}^{(n)} b_{j-l}^{(n)}\to 0,\quad l\le \min\{j,m\}, \qquad \mathtt{T}_j^{(n)}\to T_j,\quad 1\le j\le k
    \]
    a.s.~and in $L^2$ as $n\to\infty$.

    We may now pass to the limit. Let $\mathbf{T}=(T_1,\cdots,T_k)$, by \eqref{eq:remove_Poisson} and \eqref{eq:remove_Poisson2} we have
    \begin{align*}
        \ev[\ol T_jF(\mathbf{T})]  &= \lim_{n\to\infty}\ev[\ol {\mathtt{T}}_j^{(n)}F(\mathbf{T}^{(n)})]\\
        &=\lim_{n\to\infty}-i\frac{2}{\beta j}\ev\left[\{F(\mathbf{T}^{(n)}),\overline{\mathtt{T}}_j^{(n)}\}_n\right] \\
        &=\lim_{n\to\infty}\frac{2}{\beta j}\sum_{m=1}^k \ev\left[mj\Big(\delta_{mj}-\sum_{l=1}^{\min\{m,j\}}\bar b_{m-l}^{(n)}b_{j-l}^{(n)}\Big)\partial_{ z_m}F(\mathbf{T}^{(n)})\right]\\
        &=\frac{2j}{\beta}\ev\left[\partial_{z_j}F(\mathbf{T})\right]
    \end{align*}
    as desired. Applying the same argument to the complex conjugates gives that 
    \begin{align*}
        \ev[T_j F(\mathbf{T}) ]  = \frac{2j}{\beta}\ev[\partial_{\bar z_j}F(\mathbf{T})]. 
    \end{align*}
    Thus, by Stein's lemma in Proposition \ref{prop:Gauss-Stein}, we conclude that $T_j,1\le j\le k$ are independent random variables such that $T_j\sim\CC N(0,\frac{2j}{\beta})$. This finishes the proof.
\end{proof}
We now turn to the proof of Claim \ref{claim:D}. 
\begin{proof}[Proof of Claim \ref{claim:D}]
Recall that $\mathtt{T}_k^{(n)}=-k\ol{\mathtt{t}}_k^{(n)}$, where $\mathtt{t}_k^{(n)}$ is the $k$-th Taylor coefficient of $\log \Phi_n^*$. We claim that 
\begin{align}\label{eq:D-Phi}
\cD_n\Phi_n^*(z)  = z\partial_z\Phi_n^*(z),\qquad \cD \Phi_n(z) = (z\partial_z -n) \Phi_n.
\end{align}
From this, it follows that 
\[
\cD_n \log \Phi_n^*(z) = \frac{\cD_n\Phi_n^*(z)}{\Phi_n^*(z)} = \frac{z\partial_z \Phi_n^*(z)}{\Phi_n^*(z)} = z\partial_z \log\Phi_n^*(z).
\]
By comparing the Taylor coefficients again, we obtain that
\[
\cD_n \mathtt{t}_k^{(n)}=[z^{k-1}](\partial_z\log \Phi_n^*(z)) = k\mathtt{t}_k^{(n)},
\]
which implies $\cD_n \ol{\mathtt{T}}_k^{(n)}=-k\ol{\mathtt{T}}_k^{(n)}$. 

It remains to prove \eqref{eq:D-Phi}. We proceed by induction on $n$. The case when $n=0$ is trivial. Assuming the identities \eqref{eq:D-Phi} hold for $k\le n$. Then, by Szeg\H o recursion \eqref{eq:Szego1} we have
\begin{align*}
    \cD_{n+1}\Phi_{n+1}^*(z) &= \cD_{n+1}(\Phi_n^*(z)-\alpha_n z\Phi_n(z))\\
    &= \cD_n\Phi_n^*(z)- z\Phi_n(z)\cD_{n+1}(\alpha_n) -\alpha_n z\cD_n\Phi_n(z),
\end{align*}
where we used the fact that $\Phi_n,\Phi_n^*$ only depend on $\alpha_0,\alpha_1,\cdots,\alpha_{n-1}$ in the second step. Since $\cD_{n+1}\alpha_n=(n+1)\alpha_n$, by the induction hypothesis \eqref{eq:D-Phi} we get 
\begin{align*}
     \cD_{n+1}\Phi_{n+1}^*(z) 
     % &= z\partial_z\Phi_n^*(z)-(n+1)\alpha_nz\Phi_n(z)-\alpha_n z(z\partial_z-n)\Phi_n(z)
     &=z\partial_z\Phi_n^*(z) - \alpha_nz\Phi_n(z) - \alpha_n z^2\partial_z\Phi_n(z)\\
     &=z\partial_z(\Phi_n^*(z)-\alpha_nz\Phi_n(z))=z\partial_z\Phi_{n+1}^*(z)
\end{align*}
as claimed. Repeating the same argument gives the second identity in \eqref{eq:D-Phi}. This completes the induction, and hence the proof. 
\end{proof}
We note that, by standard OPUC theory, the mapping
\begin{align}\label{eq:rotation}
    \alpha_j\mapsto e^{-i(j+1)\theta}\alpha_j,\qquad j\ge 0,
\end{align}
corresponds to rotating the probability measure determined by $\bm{\alpha}=(\alpha_0,\alpha_1,\cdots)$ by the angle $\theta$. Denote the transformed sequence by $\bm{\alpha}(\theta)$ with $\alpha_j(\theta)=e^{-i(j+1)\theta}\alpha_j$, then one can check that 
\[
\Phi_n^*(z,\bm{\alpha}(\theta)) = \Phi_n^*(e^{-i\theta}z,\bm{\alpha}),
\]
where $\Phi_n^*(\cdot,\bm{\alpha})$ denotes the $n$-th reversed polynomial associated with the sequence of Verblunsky coefficients $\bm{\alpha}$. It follows that $\mathtt{t}_k^{(n)}({\bm \alpha}(\theta)) = e^{-ik\theta}\mathtt{t}_k^{(n)}({\bm\alpha})$. Then the differential operator $-i\cD_n$ is just the infinitesimal generator of the action \eqref{eq:rotation}, and the statement \eqref{eq:D-Phi} follows from a direct differentiation at $\theta=0$.

We conclude the paper with two remarks on the limiting objects and the rate of convergence. 
\begin{remark}\label{rmk:infinite}
Using the infinitesimal generator argument, one may extend the definition of $\cD_n$ to the infinite dimensional setting
\[
\cD_\infty = \sum_{j=-1}^\infty (j+1)(\alpha_j\partial_{\alpha_j}-\bar\alpha_j\partial_{\bar\alpha_j}),
\]
and prove that  $\mathcal{D}_\infty\ol{T}_k = k\ol{T}_k$. Then the variables $t_k$ can be viewed as canonical coordinates or 'orthonormal eigenvector' for the infinite Ablowitz-Ladik Poisson bracket, in the sense that 
\[
\{t_j,t_k\}_\infty = 0,\qquad \{\bar t_j,t_k\}_\infty=i\delta_{jk}.
\]
Consequently, we have $\{F({\bf T}),\ol{T}_k\}_\infty = ik^2\partial_{z_k}F({\bf T})$ assuming the infinite AL Poisson bracket converges. The exact integration-by-parts formulas for the Verblunsky coefficients in Proposition \ref{prop:beta-Stein} then allow us to connect $\ev[\{F({\bf T}),\ol{T}_k\}_\infty]$ to $\ev[F({\bf T})\ol{T}_k]$. The Gaussian limits then follow from Stein's method. 
\end{remark}

\begin{remark}
Using our method, we expect the convergence in Theorem \ref{thm:main} to remain valid when $k=k(n)$ tends to infinity. Indeed, by the estimate of $\ev[|t_k-\mathtt{t}_k^{(n)}|^2]\asymp \sum_{j\ge n}j^{-2}=O(n^{-1})$, we have $\ev[\sum_{j=1}^k|\mathtt{T}_j^{(n)}-T_j|^2]\asymp n^{-1}\sum_{j=1}^k j^2\asymp \frac{k^3}{n}$. This suggests that the convergence \eqref{eq:conv_vec} should still hold when $k=o(n^{1/3})$. However, this rate and range does not improve the result of Johansson-Lambert \cite{JohanssonLambert}, where they prove that the Wasserstein distance $W_2$ between the vector of traces and the corresponding Gaussian vectors is $O(k^2/n)$, which tends to $0$ when $k=o(n^{1/2})$. 
\end{remark}

\section{Appendix}\label{sec:Appendix}
We present the proofs of Propositions \ref{prop:beta-Stein} and \ref{prop:Poisson_id}.
\begin{proof}[Proof of Proposition \ref{prop:beta-Stein}]
    The proof follows from integration by parts. The probability density function of $\alpha$ is given by
\begin{align*}
    \frac{a}{\pi}(1-|z|^2)^{a-1}\ind_{\mathbb{D}}(z).
\end{align*}
Taking the Wirtinger derivatives gives that
\begin{align*}
    \partial_z(1-|z|^2)^{a} = -a\bar z(1-|z|^2)^{a-1},\quad\partial_{\bar a}(1-|z|^2)^{a} = -a z(1-|z|^2)^{a-1}.
\end{align*}
Hence, for any real differentiable functions $f$ and $g$ we have
\begin{align*}
     \ev[\rho^2(\partial_{\bar z}f\cdot g)(\alpha)] &= \frac{a}{\pi}\int_\D \partial_{\bar z}f(z)\cdot \left( (1-|z|^2)^{a}g(z) \right)d^2z\\
     &=-\frac{a}{\pi} \int_\D f(z) \partial_{\bar z}\left( (1-|z|^2)^{a}g(z) \right)d^2z\\
     &=a\ev[\alpha f(\alpha)g(\alpha)]-\ev[\rho^2 (f\partial_{\bar z}g)(\alpha)]. 
\end{align*}
% \begin{align}
%     \ev[\bar \alpha_j f(\alpha_j,\bar\alpha_j)] &= \frac{\beta_j}{\pi}\int_\D  \bar z\cdot f(z,\bar z) (1-|z|^2)^{\beta_j-1}d^2z\notag\\
%     &= \frac{1}{\pi}\int_\D \partial_z(1-|z|^2)^{\beta_j} \cdot f(z,\bar z)d^2z\notag\\
%     &= \frac{1}{\pi}\int_\D \partial_z f(z,\bar z) \cdot (1-|z|^2)^{\beta_j}d^2z\notag\\
%     &=\beta_j^{-1}\ev\left[\rho_j^2\cdot \partial_{\alpha_j}f(\alpha_j,\bar \alpha_j)\right]\label{eq:Stein_beta_1}.
% \end{align}
% where we denote by $\rho_j^2=1-|\alpha_j|^2$.
Rearranging the equation proves \eqref{eq:beta-Stein1}. Taking the Wirtinger derivative $\partial_{\bar z}$ and repeating the same arguments prove \eqref{eq:beta-Stein2}.
\end{proof}

\begin{proof}[Proof of Proposition \ref{prop:Poisson_id}]
    The case when $n=0$ is trivial. We proceed by induction and assume the statements hold for $k\le n$. By the definition of the Poisson bracket we have 
    \[
    \{f,g\}_{n+1} = \{f,g\}_n + i\rho_n^2(\partial_{\bar \alpha_n}f\partial_{\alpha_n}g - \partial_{\alpha_n}f\partial_{\bar\alpha_n}g).
    \]
    Hence, by the Szeg\H o recursion \eqref{eq:Szego1}  we get 
    \begin{align*}
        \{\Phi_{n+1}(z),\Phi_{n+1}(w)\}_{n+1} &= \{z\Phi_n(z)-\bar\alpha_n\Phi_n^*(z),w\Phi_n(w)-\bar\alpha_n\Phi_n^*(w)\}_{n+1}\\
        &=-\bar\alpha_n \left(z\{\Phi_n(z),\Phi_n^*(w)\}_{n+1}-w\{\Phi_n(w),\Phi_n^*(z)\}_{n+1}\right)=0,
    \end{align*}
    where the last step follows from the induction hypothesis \eqref{id:Poisson2} and $\{\Phi_n,\Phi_n^*\}_{n+1}=\{\Phi_n,\Phi_n^*\}_{n}$. Similarly, we have $\{\Phi_{n+1}^*(z),\Phi_{n+1}^*(w)\}_{n+1}=0$.

    It remains to prove \eqref{id:Poisson2} at $n+1$. By Szeg\H o recursion \ref{eq:Szego1} again, we have
    \begin{align}
        \{\Phi_{n+1}(z),\Phi_{n+1}^*(w)\}_{n+1} &= \{z\Phi_n(z)-\bar \alpha_n\Phi_n^*(z), \Phi_{n}^*(w)-\alpha_nw\Phi_n(w)\}_{n+1}\notag\\
        &=z\{\Phi_n(z),\Phi_n^*(w)\}_n-|\alpha_n|^2 w\{\Phi_n(w),\Phi_n^*(z)\}_n+ w\Phi_n^*(z)\Phi_n(w)\{\bar\alpha_n,\alpha_n\}_{n+1}\notag\\
        &=\rho_n^2 \frac{izw}{z-w}(\Phi_n(z)\Phi_n^*(w)-\Phi_n(w)\Phi_n^*(z))+i\rho_n^2w\Phi_n^*(z)\Phi_n(w)\notag\\
        % &=\frac{izw}{z-w}\rho_n^2\Phi_n(z)\Phi_n^*(w)-\frac{iw^2}{z-w}\Phi_n(w)\Phi_n^*(z)\\
        &=\frac{iw}{z-w}\rho_n^2(z\Phi_n(z)\Phi_n^*(w)-w\Phi_n(w)\Phi_n^*(z))\label{eq:indunction}. 
    \end{align}
    On the other hand, the Szeg\H o recursion gives that
    \[ 
    \Phi_{n+1}(z)\Phi_{n+1}^*(w)-\Phi_{n+1}(w)\Phi_{n+1}^*(z)=\rho_n^2(z\Phi_n(z)\Phi_n^*(w)-w\Phi_n(w)\Phi_n^*(z)).
    \]
    Plugging into the right hand side of \eqref{eq:indunction} completes the induction and the proof.  
\end{proof}

% \begin{proposition}\label{prop:rotate} Let $\mu$ be a probability measure on $\partial\D$ with respect to the Verblunsky coefficients $\alpha_{0},\alpha_1\cdots$. Then the Verblunsky coefficients associated with $ \mu^\theta:=e^{i\theta}\mu$ are given that $\alpha^\theta_{j}=e^{-i(j+1)s}\alpha_j$ for $j\ge -1$.
% \end{proposition}
% % \begin{proof}
% % 	This follows from (1.5.80) in \cite{OPUC1}. We have
% % 	\[
% % 	\alpha_k = (-1)^k\frac{D_k(e^{-i\theta}d\mu)}{D_k(d\mu)},\quad 	\alpha_k' = (-1)^k\frac{D_k(e^{-i\theta}d\mu')}{D_k(d\mu')},
% % 	\]
% % 	where
% % 	\[
% % 	D_k(d\mu) = \det(T^{(k+1)}),\quad T^{(k)}=\begin{pmatrix}
% % 	c_0&c_1&\cdots&c_{k-1}\\ c_{-1}&c_0&\cdots&c_{n-2}\\\vdots&\vdots&\cdots&\vdots\\c_{1-n}& c_{2-n}&\cdots&c_0
% % 	\end{pmatrix},\quad c_j=\int e^{-ij\theta}d\mu(\theta).
% % 	\]
% % 	Since $D_k(d\mu)=D_k(d\mu')$, we have
% % 	\[
% % 	\frac{\alpha_k'}{\alpha_k} = \frac{D_k(e^{-i\theta}d\mu')}{D_k(e^{-i\theta}d\mu)}=\frac{\sum_{\sigma\in S_{k+1}}\prod_{j=1}^{k+1}T^{(k+1)}_{j,\sigma(j)}(e^{-i\theta}d\mu')}{\sum_{\sigma\in S_{k+1}}\prod_{j=1}^{k+1}T^{(k+1)}_{j,\sigma(j)}(e^{-i\theta}d\mu)}=e^{-i(k+1)s}.
% % 	\]
% % \end{proof} 

\bibliographystyle{plain}
\bibliography{rmt}

 {{  
 		\bigskip
 		\footnotesize
 		
 		\noindent Yun Li, \textsc{School of Mathematical Sciences, Nankai University, Tianjin 300071, China.}\\
        \textit{Email:} \texttt{liyun1723@hotmail.com}
        
 }}
\end{document}